\documentclass[a4paper, 12pt]{amsart}

\usepackage{txfonts}
\usepackage{longtable}
\usepackage{color}
\usepackage[utf8]{inputenc}

\newtheorem{theorem}{Theorem}[section]
\newtheorem{definition}[theorem]{Definition}
\newtheorem{proposition}[theorem]{Proposition}
\newtheorem{lemma}[theorem]{Lemma}
\newtheorem{corollary}[theorem]{Corollary}

\newtheorem{remark}[theorem]{Remark}
\newtheorem{example}[theorem]{Example}

\newcommand{\C}{\mathbb C}

\newcommand{\Z}{\mathbb Z}

\newcommand{\SL}{\mathit{SL}}
\begin{document}

\title
{A polynomial defined by the $\SL(2;\C)$-Reidemeister torsion for a homology 3-sphere obtained by a Dehn surgery along a $(2p,q)$-torus knot}

\author{Teruaki Kitano}

\address{Department of Information Systems Science, 
Faculty of Science and Engineering, 
Soka University, 
Tangi-cho 1-236, 
Hachioji, Tokyo 192-8577, Japan}

\email{kitano@soka.ac.jp}

\thanks{2010 {\it Mathematics Subject Classification}. 57M27.}

\thanks{{\it Key words and phrases.\/}
Reidemeister torsion, 
a torus knot, Brieskorn homology 3-sphere, 
$\SL(2;\C)$-representation. }

\begin{abstract}
Let $K$ be a $(2p,q)$-torus knot. 
Here $p$ and $q$ are coprime odd positive integers. 
$M_n$ is a 3-manifold obtained by $\frac1n$-Dehn surgery along $K$. 
We consider a polynomial $\sigma_{(2p,q,n)}(t)$ 
whose zeros are the inverses of the Reideimeister torsion of $M_n$ for $\SL(2;\C)$-irreducible representations under some normalization. 
Johnson gave a formula for the case of the $(2,3)$-torus knot 
under another normalization. 
We generalize this formula for a $(2p,q)$-torus knot by using Tchebychev polynomials. 
\end{abstract}

\maketitle
\section{Introduction}

Reidemeister torsion is a piecewise linear invariant for manifolds. 
It was originally defined by Reidemeister, Franz and de Rham in 1930's. 
In 1980's Johnson \cite{Johnson} developed a theory of the Reidemeister torsion 
from the view point of relations to the Casson invariant. 
He also derived an explicit formula for the Reidemeister torsion of homology 3-spheres 
obtained by $\frac1n$-Dehn surgeries along a torus knot for $\SL(2;\C)$-irreducible representations. 

Let $K$ be a $(2p, q)$-torus knot where $p,q$ are coprime, positive odd integers.  
Let $M_n$ be a closed 3-manifold obtained by a $\frac1n$-surgery along $K$. 
We consider the Reidemeister torsion $\tau_\rho(M_n)$ of $M_n$ for an irreducible representation $\rho:\pi_1(M_n)\rightarrow\SL(2;\C)$. 
 
Johnson gave a formula for any non trivial values of $\tau_{\rho}(M_n)$. 
Furthermore in the case of the trefoil knot,  
he proposed to consider 
the polynomial whose zero set coincides 
with the set of all non trivial values $\{\frac{1}{\tau_\rho(M_n)}\}$, 
which is denoted by ${\sigma}_{(2,3,n)}(t)$.
Under some normalization of ${\sigma}_{(2,3,n)}(t)$,  
he gave a 3-term relation 
among ${\sigma}_{(2,3,n+1)}(t),{\sigma}_{(2,3,n)}(t)$ and ${\sigma}_{(2,3,n-1)}(t)$ 
by using Tchebychev polynomials. 

In this paper we consider one generalization of this polynomial 
for a $(2p, q)$-torus knot. 
Main results of this paper are Theorem 4.3 and Proposition 5.1.

\flushleft{Acknowledgements.}
This research was partially supported by JSPS KAKENHI 25400101. 
\section{Definition of Reidemeister torsion}

First let us describe definitions and properties 
of the Reidemeister torsion for $\SL(2;\C)$-representations. 
See Johnson \cite{Johnson}, Kitano \cite{Kitano94-1,Kitano94-2} and Milnor \cite{Milnor61, Milnor62, Milnor66} for details.

Let $W$ be a $d$-dimensional vector space over $\C$ 
and let ${\bf b}=(b_1,\cdots,b_d)$ and
${\bf c}= (c_1,\cdots,c_d)$ be two bases for $W$. 
Setting $b_i=\sum p_{ji}c_{j}$, 
we obtain a nonsingular matrix $P=(p_{ij})$ with entries in $\C$. 
Let $[{\bf b}/{\bf c}]$ denote the determinant of $P$.

Suppose
\[
C_*: 0\rightarrow C_k
\overset{\partial_k}{\rightarrow} 
C_{k-1}
\overset{\partial_{k-1}}{\rightarrow} 
\cdots 
\overset{\partial_{2}}{\rightarrow}
C_1\overset{\partial_1}{\rightarrow} C_0\rightarrow 0
\]
is an acyclic chain complex of finite dimensional vector spaces over $\C$. 
We assume that a preferred basis ${\bf c}_i$ for $C_i$ is given for each $i$. 
Choose some basis ${\bf b}_i$ for $B_i=\mathrm{Im}(\partial_{i+1})$ 
and take a lift of it in $C_{i+1}$, which we denote by $\tilde{\bf b}_i$.
Since $B_i=Z_i=\mathrm{Ker}{\partial_{i}}$, 
the basis ${\bf b}_i$ can serve as a basis for $Z_i$. 
Furthermore since the sequence
\[
0\rightarrow Z_i
\rightarrow C_i
\overset{\partial_{i}}{\rightarrow} B_{i-1}\rightarrow 0
\]
is exact, the vectors $({\bf b}_i,\tilde{\bf b}_{i-1}) $ form a basis for $C_i$. 
Here $\tilde{\bf b}_{i-1}$ is a lift of ${\bf b}_{i-1}$ in $C_i$. 
It is easily shown that
$[{\bf b}_i,\tilde{\bf b}_{i-1}/{\bf c}_i]$ does not depend on the choice 
of a lift $\tilde{\bf b}_{i-1}$. 
Hence we can simply denote
it by $[{\bf b}_i, {\bf b}_{i-1}/{\bf c}_i]$.

\begin{definition}
The torsion of the chain complex $C_*$ is given by the alternating product 
\[
\prod_{i=0}^k[{\bf b}_i, {\bf b}_{i-1} /{\bf c}_i]^{(-1)^{i+1}}
\]
and we denote it by $\tau(C_*)$.
\end{definition}

\begin{remark}
It is easy to see that $\tau(C_\ast)$ does not depend on the choices of the bases 
$\{{\bf b}_0,\cdots,{\bf b}_k\}$. 
\end{remark}

Now we apply this torsion invariant of chain complexes to the following geometric situations. 
Let $X$ be a finite CW-complex and $\tilde X$ a universal covering of $X$. 
The fundamental group $\pi_1 X$ acts on $\tilde X$ from the right-hand side as deck transformations. 
Then the chain complex $C_*(\tilde{X};\Z)$ has the structure of a chain complex of free $\Z[\pi_1 X]$-modules. 

Let $\rho:\pi_1 X\rightarrow \SL(2; \C)$  be a representation. 
We denote the 2-dimensional vector space $\C^2$ by $V$. 
Using the representation $\rho$, $V$ admits the structure of a $\Z[\pi_1 X]$-module 
and then we denote it by $V_\rho$. 
Define the chain complex $C_*(X; V_\rho)$ 
by $C_*({\tilde X}; \Z)\otimes_{\Z[\pi_1 X]} V_\rho$ 
and choose a preferred
basis
\[
(\tilde{u}_1\otimes {\bf e}_1, \tilde{u}_1\otimes {\bf e}_2, \cdots,\tilde{u}_d\otimes{\bf e}_1, \tilde{u}_d\otimes{\bf e}_2)
\]
of $C_i(X; V_\rho)$ where $\{{\bf e}_1 , {\bf e_2}\}$ is a canonical basis of $V=\C^2$ 
and 
$u_1,\cdots,u_d$ are the $i$-cells 
giving a basis of $C_i(X; \Z)$. 

Now we suppose that $C_*(X; V_\rho)$ is acyclic, 
namely all homology groups $H_*(X; V_\rho)$ are vanishing. 
In this case we call $\rho$ an acyclic representation. 

\begin{definition}
Let $\rho:\pi_1(X)\rightarrow \SL(2; \C)$ be an acyclic representation. 
Then the Reidemeister torsion $\tau_\rho(X)\in\C\setminus\{0\}$ is defined by the torsion $\tau(C_*(X; V_\rho))$ of $C_*(X; V_\rho)$. 
\end{definition}

\begin{remark}
\noindent
\begin{enumerate}
\item
We define $\tau_\rho(X)=0$ for a non-acyclic representation $\rho$.
\item
The definition of $\tau_\rho(X)$ depends on several choices. 
However it is well known that the Reidemeister torsion is a piecewise linear invariant for $X$ with $\rho$. 
\end{enumerate}
\end{remark}

Now let $M$ be a closed orientable 3-manifold 
with an acyclic representation $\rho:\pi_1(M)\rightarrow \SL(2; \C)$.  
Here we take a torus decomposition of $M=A\cup_{T^2} B$. 
For simplicity, we write the same symbol $\rho$ 
for restricted representations to images 
of $\pi_1(A),\pi_1(B),\pi_1(T^2)$ in $\pi_1(M)$ by inclusions. 

By this decomposition, 
we have the following formula. 

\begin{proposition}
Let $\rho:\pi_1(M)\rightarrow \SL(2;\C)$ a representation. 
Assume all homogy groups $H_{\ast}(T^{2};V_{\rho})=0$. 
Then 
all homology groups $H_\ast(M;V_\rho)=0$ 
if and only if 
both of all homology groups $H_\ast(A;V_{\rho})=H_\ast(B;V_{\rho})=0$. 
In this case, it holds 
\[
\tau_\rho(M)=\tau_{\rho}(A)\tau_{\rho}(B).
\]
\end{proposition}

\section{Johnson's theory}

We apply the above proposition to a 3-manifold obtained by Dehn-surgery along a knot. 
Now let $K\subset S^3$ be a $(2p,  q)$-torus knot with coprime odd integers$p,q$. 
Further let $N(K)$ be an open tubular neighborhood of $K$ 
and $E(K)$ its knot exterior $S^3\setminus {N}(K)$. 
We denote its closure of ${N(K)}$ by $\overline{N}$ which is homeomorphic to $S^{1}\times D^{2}$.
Now we write $M_n$ to a closed orientable 3-manifold obtained by a $\frac1n$-surgery along $K$. 
Naturally there exists a torus decomposition $M_n=E(K)\cup \overline{N}$ of $M_n$. 

\begin{remark}
This manifold $M_n$ is diffeomorphic to a Brieskorn homology 3-sphere $\Sigma(2p,q,N)$ 
where $N=|2pq n+1|$. 
\end{remark}
 
Here the fundamental group of $E(K)$ has a presentation as follows. 
\[
\pi_1(E(K))=\pi_1 (S^3\setminus K)=\langle x,y\ |\ x^{2p} =y^ q\rangle
\]
Furthermore 
the fundamental group $\pi_1(M_n)$ admits the presentation as follows;
\[
\pi_1(M_n)=\langle x,y\ |\ x^{2p} =y^q, m l^n= 1\rangle
\]
where $m=x^{-r}y^s\ (r,s\in\Z,\ 2p s- q r=1)$ is a meridian of $K$ 
and $l=x^{-2p}m^{2p q}=y^{- q}m^{2p q}$ is similarly a longitude. 

Let $\rho:\pi_1(E(K))=\pi_1(S^3\setminus K)\rightarrow\SL(2;\C)$ a representation. 
It is easy to see a given representation $\rho$ can be extended to $\pi_1(M_n)\rightarrow\SL(2;\C)$ as a representation 
if and only if $\rho(ml^n)=E$. 
Here $E$ is the identity matrix in $\SL(2;\C)$. 
In this case by applying Proposition 2.5, 
\[
\tau_\rho(M_n)={\tau_\rho(E(K))}{\tau_\rho(\overline{N})}
\]
for any acyclic representation $\rho:\pi_1(M_n)\rightarrow\SL(2;\C)$. 

Now we consider only irreducible representations of $\pi_1(M_n)$, 
which is extended from the one on $\pi_1(E(K))$. 
It is seen that the set of the conjugacy classes of the $\SL(2 ;\C)$-irreducible
representations is finite. 
Any conjugacy class can be represented by $\rho_{(a,b,k)}$ for some $(a,b,k)$ 
such that
\begin{enumerate}
\item
$0<a<2p,0<b< q, a\equiv b \ \text{mod } 2$,
\item
$0<k<N=|2p q n+1|, k\equiv na\  \text{mod } 2$,
\item
$\mathrm{tr}(\rho_{(a,b,k)}(x))=2\cos \frac{a\pi }{2p}$,
\item
$\mathrm{tr} (\rho_{(a,b,k)}(y))=2 \cos\frac{b\pi}{ q}$,
\item
$\mathrm{tr} (\rho_{(a ,b,k)}(m)) =2 \cos\frac{k\pi}{N}$.
\end{enumerate}

Johnson computed $\tau_{\rho_{(a,b,k)}}(M_{n})$ as follows.

\begin{theorem}[Johnson]
\noindent
\begin{enumerate}
\item
A representation $\rho_{(a,b,k)}$ is acylic if and only if 
$a\equiv b\equiv 1, k \equiv n \text{ mod }2$.
\item
For any acyclic representation ${\rho_{(a,b,k)}}$ with $a\equiv b{\equiv} 1, k\equiv n \text{ mod }2$, 
then 
\[
\tau_{\rho_{(a,b,k)}}(M_{n})
=\frac{1}{2\left(1-\cos\frac{a\pi}{2p}\right) 
\left(1-\cos\frac{b\pi}{ q}\right)
\left(1+\cos\frac{{2}p q k\pi }N\right)}.
\]
\end{enumerate}
\end{theorem}

\begin{remark}
\noindent
\begin{itemize}
\item
In fact Johnson proved this theorem for any torus knot, not only for a $(2p,q)$-torus knot. 
\item
This Johnson's result was generalized for any Seifert fiber manifold in \cite{Kitano94-1}. 
Please see \cite{Kitano94-1} as a reference. 
\item
In general, it is not true that the set of $\{\tau_\rho(M_n)\}$ is finite. 
There exists a manifold whose Reidemsiter torsion can be variable continuously. 
Please see \cite{Kitano94-2}.
\end{itemize}
\end{remark}

Here assume $K=T(2,3)$ is the trefoil knot. 
By considering the set of non trivial values $\tau_{\rho}(M_{n})$ for irreducible representation $\rho:\pi_{1}(M_{n})\rightarrow\SL(2;\C)$, 
Johnson defined the polynomial $\bar{\sigma}_{(2, 3,n)}(t)$ 
of one variable $t$ whose zeros are the set of $\{\frac{1}{\frac12{\tau_\rho(M_{n})}}\}$. 
which is well defined up to multiplications of nonzero constants. 

\begin{theorem}[Johnson]
Under normalization by $\bar{\sigma}_{(2,3,n)}(0)=(-1)^n$, 
there exists the 3-term relation s.t. 
\[
\bar{\sigma}_{(2,3,n+1)}(t)=(t^3-6t^2+9t-2)\bar{\sigma}_{(2,3,n)}(t)
-\bar{\sigma}_{(2,3,n-1)}(t).
\]
\end{theorem}

\begin{remark}
The polynomial $t^3-6t^2+9t-2$ is given by $2T_6\left(\frac12\sqrt{t}\right)$. 
Here $T_6(x)$ is the sixth Tchebychev polynomial. 
\end{remark}

Recall the $n$-th Tchebychev polynomial $T_n(x)$ of the first kind can be defined 
by expressing $\cos n\theta$ as a polynomial in $\cos\theta$. 
We give a summary of these polynomials. 

\begin{proposition}
The Tchebychev polynomials have following properties.
\begin{enumerate}
\item
$T_0(x)=1,T_1(x)=x$.
\item
$T_{-n}(x)=T_n(x)$.
\item
$T_n(1)={1},T_n(-1)=(-1)^n$.
\item
$T_n(0)=\begin{cases}
& 0\text{ if $n$ is odd,}\\
& (-1)^{\frac n2}\text{ if $n$ is even.}
\end{cases}$
\item
$T_{n+1}(x)=2xT_n-T_{n-1}(x)$.
\item
The degree of $T_n(x)$ is $n$.
\item
{$2T_m(x)T_n(x)=T_{m+n}(x)+T_{m-n}(x)$.}
\end{enumerate}
\end{proposition}

He we put a short list of  $T_n(x)$.
\begin{itemize}
\item
$T_0(x)=1$,
\item
$T_1(x)=x$,
\item
$T_2(x)=2x^2-1$,
\item
$T_3(x)=4x^3-3x$,
\item
$T_4(x)=8x^4-8x^2+1$,
\item
$T_5(x)=16x^5-20x^3+5x$,
\item
$T_6(x)=32x^6-48x^4+18x^2-1$.
\end{itemize}

\begin{example}
Put $p=1, q=3$ and $n=-1$. 
Then $N=|2\cdot 3\cdot(-1)+1|=5$ and $M_{-1}=\Sigma(2,3,5)$. 
In this case, it is easy to see that 
$a=b=1$ and $k=1,3$. 
By the above formula, we obtain 
\[
\begin{split}
\tau_{\rho_{(1,1,k)}}(M_{-1})
&=\frac{1}{{2\left(1-\cos\frac{\pi}2\right) 
\left(1-\cos\frac{\pi}3\right)}{\left(1+\cos\frac{6k\pi}5\right)}}\\
&=\frac{1}{{2(1-0)(1-\frac12)(1+\cos\frac{6k\pi}5)}}\\
&=\frac{1}{{1+\cos\frac{6k\pi}5}}\\
&={3\pm \sqrt{5}}.
\end{split}
\]
Hence 
we have two non trivial values of $\frac{1}{\frac{1}{2}\tau_{\rho}(M_{-1})}$ as 
\[
\begin{split}
\frac{1}{\frac{1}{2}\tau_{\rho}(M_{-1})}
&=\frac{1}{\frac{3\pm\sqrt{5}}{2}}\\
&=\frac{2}{3\pm\sqrt{5}}\\
&=\frac{3\pm\sqrt{5}}{2}.
\end{split}
\]
Therefore we have 
\[
\left(t-\left(\frac{3-\sqrt{5}}{2}\right)\right)
\left(t-\left(\frac{3+\sqrt{5}}{2}\right)\right)
=t^{2}-3t+1.
\]
Under Johnson's normalization $\bar{\sigma}_{(2,3,-1)}(0)=-1$, 
\[
\bar{\sigma}_{(2,3,-1)}(t)=-t^{2}+3t-1.
\]
Next put $n=1$. 
In this case 
\[
\begin{split}
\tau_{\rho_{(1,1,k)}}(M_{1})
&=\frac{1}{{2\left(1-\cos\frac{\pi}2\right) \left(1-\cos\frac{\pi}3\right)}{\left(1+\cos\frac{6k\pi}{7}\right)}}\\
&=\frac{1}{\left(1+\cos\frac{6k\pi}5\right)}.
\end{split}
\]
We can see as 
\[
\begin{split}
&\left(t-2{\left(1+\cos\frac{6\pi}5\right)}\right)\left(t-2{\left(1+\cos\frac{6\cdot 3\pi}5\right)}\right)\left(t-2{\left(1+\cos\frac{6\cdot 5\pi}5\right)}\right)\\
=&t^{3}-5t^{2}+6t-1\\
=&\bar{\sigma}_{(2,3,1)}(t).
\end{split}
\]
On the other hand, 
by using Johnson's formula 
\[
\begin{split}
(t^{3}-6t^{2}+9t-2)\bar{\sigma}_{(2,3,0)}(t)-\bar{\sigma}_{(2,3,-1)}(t)
=&(t^{3}-6t^{2}+9t-2)\cdot 1-(-t^{2}+3t-1)\\
=&t^{3}-5t^{2}+6t-1,
\end{split}
\]
we obtain the same polynomial. 
\end{example}

\section{Main theorem}

From this section, we consider the generalization for a $(2p, q)$-torus knot. 
Here $p,q$ are coprime odd integers. 
In this section 
we give a formula of the torsion polynomial 
${\sigma}_{(2p, q,n)}(t)$ 
for $M_n=\Sigma(2p, q,N)$ 
obtained by a $\frac1n$-Dehn surgery along $K$. 
Although Johnson considered the inverses of the half of $\tau_{\rho}(M_{n})$, 
we simply treat torsion polynomials as follows. 

\begin{definition}
A one variable polynomial ${\sigma}_{(2p, q,n)}(t)$ is called the torsion polynomial 
of $M_n$ if the zero set coincides with the set 
of all non trivial values $\{\frac{1}{\tau_\rho(M_n)}\}$ 
and it satisfies the following normalization condition 
as $\sigma_{(2p,q,n)}(0)=(-1)^{\frac{np(q-1)}{2}}$. 
\end{definition}

\begin{remark}
\noindent
\begin{enumerate}
\item
If $n=0$, then clearly $M_{n}=S^{3}$ with the trivial fundamental group. 
Hence we define the torsion polynomial to be trivial. 
\end{enumerate}
\end{remark}

From here assume $n\neq 0$. 
Recall Johnson's formula 
\[
\frac{1}{\tau_{\rho_{(a,b,k)}}(M_{n})}
={2\left(1-\cos\frac{a\pi}{2p}\right) 
\left(1-\cos\frac{b\pi}{ q}\right)
\left(1+\cos\frac{2 p q k\pi}N\right)} 
\]
where $0<a<2p,0<b< q, a\equiv b\equiv 1\text{ mod 2}, k\equiv n\text{ mod }2$. 
Here we put  
\[
C_{(2p, q,a,b)}=\left(1-\cos\frac{a\pi }{2p}\right)\left(1-\cos\frac{b\pi }{ q}\right) 
\]
and we have 
\[
\frac{1}{\tau_{\rho_{(a,b,k)}}(M_n)}=4C_{(2p,q,a,b)}\cdot \frac12\left(1+\cos\frac{2 p q k\pi}N\right).
\]

Main result is the following. 

\begin{theorem}
The torsion polynomial of $M_{n}$ is given by 
\[
{\sigma}_{(2p, q,n)}(t)=
\prod_{(a,b)}Y_{(n,a,b)}(t)
\]
where
\[
Y_{(n, a,b)}(t)
=
\begin{cases}
&
2C_{(2p, q,a,b)}\frac{
T_{{N+1}}\left(\frac{\sqrt{t}}{2\sqrt{C_{(2p, q,a,b)}}}\right)
-T_{{N-1}}\left(\frac{\sqrt{t}}{2\sqrt{C_{(2p, q,a,b)}}}\right)}{t-2C_{(2p, q,a,b)}}\hskip 0.3cm (n>0)\\
&
-2C_{(2p, q,a,b)}\frac{
T_{{N+1}}\left(\frac{\sqrt{t}}{2\sqrt{C_{(2p, q,a,b)}}}\right)
-T_{{N-1}}\left(\frac{\sqrt{t}}{2\sqrt{C_{(2p, q,a,b)}}}\right)}{t-2C_{(2p, q,a,b)}}\ (n<0).\\
\end{cases}
\]
Here $N=|2pqn+1|$ 
and a pair of integers $(a,b)$ is satisfying the following conditions;
\begin{itemize}
\item
$0<a<2p,0<b< q$, 
\item
$a\equiv b\equiv 1\text{ mod }2$, 
\item
$0<k<N, k\equiv n\text { mod }2$.
\end{itemize}
\end{theorem}

\begin{proof}
\noindent
\flushleft{Case 1:$n> 0$}

We modify one factor $(1+\cos\frac{2p q k\pi}N)$ 
of $\displaystyle\frac{1}{\tau_\rho(M_n)}$ as follows. 

\begin{lemma}
The set $\{\cos\frac{2p q  k\pi}{N}\ |\ 0<k<N, k\equiv n\text{ mod }2\}$ is equal to  
the set $\{\cos\frac{2p k\pi }{N}\ |\ 0<k<\frac{N}{2}\}$. 
\end{lemma}

\begin{proof}
Now $N=2p q n+1$ is always an odd integer. 

For any $k>\frac{N}{2}$, then clearly $N-k<\frac{N}{2}$. 
Then 
\[
\begin{split}
\cos\frac{2p q (N-k)\pi}{N}
&=\cos\left(2p q\pi -\frac{2p q  k\pi}{N}\right)\\
&=(-1)^{2p q}\cos\left(-\frac{2p q k\pi}{N}\right)\\
&=\cos\left(\frac{2p q  k\pi}{N}\right).
\end{split}
\]
Here if $k$ is even (resp. odd), then $N-k$ is odd (resp. even). 
Hence it is seen 
\[
\{\cos\frac{2 p q k\pi}{N}
\ |\ 0<k<N, k\equiv n\text{ mod }2\}=
\{\cos\frac{2 p q k\pi}{N}\ |\ 0<k<\frac{N}{2}\}. 
\]
For any $k<\frac{N}{2}$, 
there exists uniquely $l$ such that $-\frac{N}{2}< l< \frac{N}{2}$ 
and $l\equiv qk$ mod $N$. 
Further there exists uniquely $l$ such that $0< l< \frac{N}{2}$ and $l\equiv \pm qk$ mod $N$. 
Here 
$\cos\frac{2p q k\pi }{N}=\cos\frac{2p l\pi}{N}$ 
if and only if $2p q k\equiv \pm 2p l$ mod $N$. 
Therefore it is seen that the set 
\[
\left\{\cos\frac{2pq k\pi}{N}\ |\ 0<k<\frac{N}{2}\right\}
=
\left\{\cos\frac{2p k\pi}{N}\ |\ 0<k<\frac{N}{2}\right\}. 
\]
\end{proof}

Now we can modify  
\[
\begin{split}
\frac12\left(1+\cos\frac{2p k\pi}{N}\right)
&=\frac12\cdot2\cos^2\frac{2p k\pi}{2N}\\
&=\cos^2\frac{p k\pi}{N}.
\end{split}
\]

We put 
\[
z_k=\cos\frac{p k\pi}{N}\ (0<k<N) 
\]
and substitute $x=z_k$ to $T_{N+1}(x)$. 
Then it holds 
\[
\begin{split}
T_{N+1}(z_k)
&=\cos\left(\frac{(N+1)(p k\pi)}{N}\right)\\
&=\cos\left(p k\pi+\frac{pk\pi }{N}\right)\\
&=(-1)^{p k}z_k .
\end{split}
\]
Similarly it is seen   
\[
\begin{split}
T_{N-1}(z_k)
&=\cos\left(\frac{(N-1)(p k\pi)}{N}\right)\\
&=\cos\left(p k\pi-\frac{p k\pi}{N}\right)\\
&=(-1)^{p k}z_k.
\end{split}
\]
Hence it holds 
\[
T_{N+1}(z_k)-T_{N-1}(z_k)=0.
\]
By properties of Tchebyshev polynomials, 
it is seen that  
\begin{itemize}
\item
$T_{N+1}(1)-T_{N-1}(1)=0$,
\item
$T_{N+1}(-1)-T_{N-1}(-1)=0$.
\end{itemize}

Therefore we consider the following;
\[
X_n(x)=
\begin{cases}
&
\frac{T_{N+1}(x)-T_{N-1}(x)}{2(x^2-1)}\hskip 0.3cm (n>0)\\
&
-\frac{T_{N+1}(x)-T_{N-1}(x)}{2(x^2-1)}\ (n<0).\\
\end{cases}
\]

We mention that the degree of $X_n(x)$ is $N-1$. 

By the above computation, 
$z_1,\cdots, z_{N-1}$ are the zeros of $X_n(x)$. 
Further we can see 
\[
\begin{split}
z_{N-k}
&=\cos\frac{p(N-k)\pi}{N}\\
&=\cos(p\pi-\frac{ pk\pi}{N})\\
&=(-1)^p\cos(-\frac{pk\pi}{N})\\
&=(-1)^p\cos(\frac{pk\pi}{N})\\
&=-z_k.
\end{split}
\]
This means $N-1$ roots $z_1,\cdots, z_{N-1}$ of $X_n(x)=0$ occur in a pairs.  
Because $T_{N+1}(x),T_{N-1}(x)$ are even functions, 
they are functions of $x^2$. 
Hence $X_n(x)$ is also an even function. 

Here by replacing $x^2$ by $\frac{t}{4C_{(2p, q,a,b)}}$, 
namely $x$ by $\frac{\sqrt{t}}{2\sqrt{C_{(2p, q,a,b)}}}$, 
we put 
\[
\begin{split}
Y_{(n,a,b)}(t)
&=X_n\left(\frac{\sqrt{t}}{2\sqrt{C_{(2p, q,a,b)}}}\right)\\
&=\frac{T_{N+1}\left(\frac{\sqrt{t}}{2\sqrt{C_{(2p, q,a,b)}}}\right)
-T_{N-1}\left(\frac{\sqrt{t}}{2\sqrt{C_{(2p, q,a,b)}}}\right)}
{2\left(\left(\frac{\sqrt{t}}{2\sqrt{C_{(2p, q,a,b)}}}\right)^2-1\right)}\\
&=\frac
{T_{N+1}\left(\frac{\sqrt{t}}{2\sqrt{C_{(2p, q,a,b)}}}\right)
-T_{N-1}\left(\frac{\sqrt{t}}{2\sqrt{C_{(2p, q,a,b)}}}\right)}
{
2\left(\frac{t}{4C_{(2p, q,a,b)}}-1\right)
}\\
&=2C_{(2p, q,a,b)}
\frac
{T_{N+1}\left(\frac{\sqrt{t}}{2\sqrt{C_{(2p, q,a,b)}}}\right)
-T_{N-1}\left(\frac{\sqrt{t}}{2\sqrt{C_{(2p, q,a,b)}}}\right)}
{
t-2C_{(2p, q,a,b)}
}\\
\end{split}
\]
Here it holds that its degree of $Y_{(n,a,b)}(s)$ is $\frac{N-1}{2}$, 
and the roots of $Y_{(n,a,b)}(t)$ 
are $4C_{(2p, q,a,b)}z_k^2=4C_{(2p, q,a,b)}\cos^2\frac{\pi k}{2p q n+1}, (0<k<\frac{N-1}{2})$, 
which are all non trivial values of $\frac{1}{\tau_{\rho_{(a,b,k)}}(M_n)}$. 
Therefore we obtain the formula. 

\flushleft{Case 2: $n<0$}

In this case we modify $N=|2pqn+1|=2pq|n|-1$. 
By the same arguments, it is easy to see the claim of the theorem can be proved. 
Therefore the proof completes. 
\end{proof}

\begin{remark}
By defining as $X_{0}(t)=1$, it implies $Y_{(0,a,b)}(t)=1$. 
Then the above statement is true for $n=0$.
\end{remark}

\begin{corollary}
The degree of $\sigma_{(2p,q,n)}(t)$ is given by $\frac{(N-1)p(q-1)}{4}$.
\end{corollary}
\begin{proof}
The number of the pairs $(a,b)$ is given by $\frac{p(q-1)}{2}$. 
As the degree of $Y_{(n,a,b)}(t)$ is $\frac{N-1}{2}$, 
then the degree of of $\sigma_{(2p,q,n)}(t)$ is given by 
$\frac{(N-1)p(q-1)}{4}$.
\end{proof}
\section{3-term relations}

Finally we prove 3-term realtions for each factor $Y_{(n,a,b)}(t)$ as follows. 

\begin{proposition}
For any $n$, it holds that 
\[
Y_{(n+1,a,b)}(t)=D(t)Y_{(n,a,b)}(t)-Y_{(n-1,a,b)}(t)
\]
where 
$D(t)=2T_{2p q }\left(\frac{\sqrt{t}}{2\sqrt{C_{2p, q,a,b}}}\right)$. 
\end{proposition}

\begin{proof}
Recall Prop. 3.2 (7); 
\[
2T_m(x)T_n(x)=T_{m+n}(x)+T_{m-n}(x).
\]
Then if $n>0$ we have 
\
\[
\begin{split}
2T_{2p q }(x)X_n(x)
&=2T_{2p q }(x)\left(\frac{T_{2p q n+2}(x)-T_{2p q n}(x)}{2(x^2-1)}\right)\\
&=\frac{(T_{2p q +2p q n+2}(x)+T_{2p q n+2-2p q }(x))-(T_{2p q +2p q n}(x)+T_{2p q n-2p q }(x))}{2(x^2-1)}\\
&=\frac{T_{2p q (n+1)+2}(x)-T_{2p q (n+1)}(x)+T_{2p q (n-1)+2}(x)-T_{2p q (n-1)}(x))}{2(x^2-1)}\\
&=X_{n+1}(x)+X_{n-1}(x).
\end{split}
\]
Therefore it can be seen that 
\[
X_{n+1}(x)=2T_{2p q }(x)X_n(x)-X_{n-1}(x)
\]
and 
\[
Y_{(n+1,a,b)}(t)
=2T_{2p q }\left(\frac{\sqrt{t}}{2\sqrt{C_{(2p,q,a,b)}}}\right)
Y_{(n,a,b)}(t)-Y_{(n-1,a,b)}(t).
\]

If n=0, 3-term relation is 
\[
Y_{(1,a,b)}(t)=D(t)Y_{(0,a,b)}(t)-Y_{(-1,a,b)}(t).
\]
It can be seen by direct computation 
\[
\begin{split}
2T_{2pq}(x)X_{0}(x)-X_{-1}(x)
&=2T_{2pq}(x)-X_{-1}(x)\\
&=X_{1}(x).
\end{split}
\]

If $n<0$, it can be also proved. 
\end{proof}

We show some examples. 
First we treat $(2,3)$-torus knot again. 

\begin{example}
Put $p=1, q=3$. 
In this case $a=b=1$. 
Then we see 
\[
C_{2,3,1,1}=\left(1-\cos\frac\pi 2\right)\left(1-\cos\frac{\pi}{3}\right)=\frac12.
\]
By applyng the theorem 4.3 and the proposition 5.1,
\[
\begin{split}
\sigma_{(2,3,-1)}(t)
&=\frac{T_{6}\left(\frac{\sqrt{t}}{\sqrt{2}}\right)
-T_{4}\left(\frac{\sqrt{t}}{\sqrt{2}}\right)}{2(1-\left(\frac{\sqrt{t}}{\sqrt{2}})^{2}\right)}\\
&=-4t^{2}+6t-1.\\
\sigma_{(2,3,0)}(t)&=1.\\
\sigma_{(2,3,1)}(t)&=8t^{3}-20t^{2}+12 t-1.
\end{split}
\]
\end{example}

We show one more example. 
\begin{example}
Here put $(2p,q)=(2,5)$. 
In this case $(a,b)=(1,1)$ or $(1,3)$ and  
the constants $C_{(2,5,1,1)}, C_{(2,5,1,3)}$ are given as follows:
\[
\begin{split}
C_{(2,5,1,1)}&=(1-\cos\frac\pi 2)(1-\cos\frac{\pi}{5})\\
&=1-\cos\frac{\pi}{5}\\
&=\frac{1}{4} \left(3-\sqrt{5}\right).
\end{split}
\]
\[
\begin{split}
C_{(2,5,1,3)}
&=(1-\cos\frac\pi 2)(1-\cos\frac{3\pi}{5})\\
&=1-\cos\frac{3\pi}{5}\\
&=\frac{1}{4} \left(3+\sqrt{5}\right).
\end{split}
\]
First we put $n=-1$. 
By Theorem 4.3, 
\[
\begin{split}
\sigma_{(2,5,-1)}(t)
&=Y_{(-1,1,1)}(t)Y_{(-1,1,3)}(t)\\
&=X_{-1}\left(\frac{\sqrt{t}}{2\sqrt{C_{2,5,1,1}}}\right)
X_{-1}\left(\frac{\sqrt{t}}{2\sqrt{C_{2,5,1,3}}}\right)\\
&=4C_{(2,5,1,1)}C_{(2,5,1,3)}
\frac{T_{10}\left(\frac{\sqrt{t}}{2\sqrt{C_{2,5,1,1}}}\right)
-T_{8}\left(\frac{\sqrt{t}}{2\sqrt{C_{2,5,1,1}}}\right)}
{t-2C_{(2,5,1,1)}}
\frac{T_{10}\left(\frac{\sqrt{t}}{2\sqrt{C_{2,5,1,3}}}\right)
-T_{8}\left(\frac{\sqrt{t}}{2\sqrt{C_{2,5,1,3}}}\right)}
{t-2C_{(2,5,1,3)}}\\
&=64 t^{10}+384 t^9-2880 t^8+5952 t^7+2336 t^6\\
&\ -14856 t^5+12192 t^4-4608 t^3+820 t^2-60 t+1.
\end{split}
\]
By the definition,  
\[
\sigma_{(2p,q,0)}(t)=1.
\]
By applying the 3-term realtion 
\[
Y_{(1,a,b)}(t)=2 T_{10}\left(\frac{\sqrt{t}}{2C_{(2,5,a,b)}}\right)Y_{(0,2p,q)}(t)
-Y_{(-1,a,b)}(t),
\]
we obtain 
\[
\begin{split}
\sigma_{(2,5,1)}(t)=
&
256 t^{12}+384 t^{11}-16064 t^{10}+61056 t^9-72000 t^8
\\
&-57888 t^7+197424 t^6-172824 t^5+273408 t^4
\\
&-16632 t^3+1880 t^2-90t+1.
\end{split}
\]
\end{example}


\end{document}